\documentclass[11pt,twoside]{amsart}

\usepackage{latexsym}
\usepackage{amsmath}
\usepackage{amsthm}
\usepackage{amssymb}
\usepackage{vmargin}
\usepackage{amscd}
\usepackage{stmaryrd}
\usepackage{euscript}
\usepackage{mathrsfs}
\usepackage{amscd}
\usepackage[all]{xy}
\usepackage{xr-hyper}

\usepackage{hyperref}

\DeclareMathAlphabet{\mathpzc}{OT1}{pzc}{m}{it}

\newtheorem{theorem}{Theorem}[section]
\newtheorem{proposition}[theorem]{Proposition}

\newtheorem{lemma}[theorem]{Lemma}
\newtheorem*{theorem*}{Theorem}
\newtheorem*{proposition*}{Proposition}
\newtheorem*{corollary*}{Corollary}
\newtheorem*{lemma*}{Lemma}
\newtheorem*{conjecture*}{Conjecture}

\theoremstyle{definition}
\newtheorem{definition}[theorem]{Definition}

\newtheorem*{definition*}{Definition}

\theoremstyle{remark}
\newtheorem{example}[theorem]{Example}
\newtheorem{examples}[theorem]{Examples}
\newtheorem{remark}[theorem]{Remark}
\newtheorem{remarks}[theorem]{Remarks}

\newtheorem{property}[theorem]{Property}
\newtheorem{properties}[theorem]{Properties}

\newtheorem*{example*}{Example}
\newtheorem*{examples*}{Examples}
\newtheorem*{remark*}{Remark}
\newtheorem*{remarks*}{Remarks}
\newtheorem*{exercise*}{Exercise}
\newtheorem*{property*}{Property}
\newtheorem*{properties*}{Properties}

\newcommand\id{\mathrm{id}}

\newcommand\ten{\otimes}

\renewcommand\H{\mathrm{H}}

\newcommand\Z{\mathbb{Z}}
\newcommand\Q{\mathbb{Q}}

\newcommand\R{\mathbb{R}}

\newcommand\bG{\mathbb{G}}

\newcommand\cA{\mathcal{A}}

\newcommand\cD{\mathcal{D}}

\newcommand\cH{\mathcal{H}}

\newcommand\cT{\mathcal{T}}

\renewcommand\O{\mathscr{O}}

\newcommand\sA{\mathscr{A}}

\newcommand\sF{\mathscr{F}}

\newcommand\sH{\mathscr{H}}

\newcommand\fX{\mathfrak{X}}
\newcommand\fY{\mathfrak{Y}}

\renewcommand\L{\Lambda}

\newcommand\Ho{\mathrm{Ho}}

\newcommand\Alg{\mathrm{Alg}}

\newcommand\Hom{\mathrm{Hom}}

\newcommand\HHom{\underline{\mathrm{Hom}}}

\newcommand\Ext{\mathrm{Ext}}

\newcommand\Top{\mathrm{Top}}

\newcommand\Spec{\mathrm{Spec}\,}

\newcommand\Set{\mathrm{Set}}

\newcommand\Aff{\mathrm{Aff}}

\newcommand\Sing{\mathrm{Sing}}

\newcommand\LLim{\varinjlim}

\newcommand\xra{\xrightarrow}

\newcommand\bt{\bullet}
\newcommand\by{\times}

\newcommand\et{\acute{\mathrm{e}}\mathrm{t}}

\newcommand\cart{\mathrm{cart}}

\newcommand\pd{\partial}

\newcommand\Hilb{\cH\mathit{ilb}}

\newcommand\cosk{\mathrm{cosk}}

\newcommand\co{\colon\thinspace}

\newcommand\oR{\mathbf{R}}

\newcommand\uleft\underleftarrow
\newcommand\uline\underline
\newcommand\uright\underrightarrow

\begin{document}

\begin{abstract}
This is an informal summary of the main concepts in \cite{stacks2}, based on notes of various seminars.  It gives constructions of  higher and derived stacks without recourse to the extensive theory developed by To\"en, Vezzosi and Lurie. Explicitly, higher stacks are described in terms of simplicial diagrams of affine schemes, which are analogous to atlases for manifolds. We also describe quasi-coherent sheaves and complexes on such objects. 
\end{abstract}

\title{Notes  characterising  higher and derived stacks concretely} 
\author{J.P.Pridham}
\thanks{This work was supported by the Engineering and Physical Sciences Research Council [grant number  EP/F043570/1].}
\maketitle

\section*{Introduction}

\subsection*{The need for simplicial objects}

Any scheme $X$ 
gives rise to a functor from rings 
to sets, sending $A$ to 
$X(A)= \Hom(\Spec A, X)$.
Likewise, any algebraic stack gives a functor from rings 
to groupoids. When $X$ is a moduli space or stack, the points of $X$ have a geometric meaning. For instance
\[
\Hilb_Y(A)= \{\text{closed subschemes of } Y \by  \Spec A, \text{ flat over } A\}.
\]

In derived algebraic geometry, the basic building blocks are  simplicial rings, or equivalently in characteristic $0$, dg rings (collectively referred to as derived rings).   Thus derived moduli spaces and stacks have to give rise to functors on derived rings $d\Alg$.
To understand what such a functor must be, we start with a derived ring $A$ 
and look at possible functors associated to $\Spec A$.
\begin{enumerate}
\item The obvious candidate is the functor 
$\Hom_{d\Alg}(A,-)\co d\Alg \to \Set$.
This is clearly no good, as it does not map quasi-isomorphisms to isomorphisms, even when $A$ is a polynomial ring.

\item This suggests the functor 
$\Hom_{\Ho(d\Alg)}(P,-)\co d\Alg \to \Set$,
where 
$\Ho(d\Alg)$ 
is the homotopy category (given by formally inverting quasi-isomorphisms). This works well for infinitesimal derived deformation theory (as in \cite{Man2}), but is not left-exact. Thus it cannot sheafify, so will not give a good global theory.

\item The solution is to look at the derived  $\Hom$-functor $\oR\HHom$, which takes values in simplicial sets (to be defined in the sequel). This maps quasi-isomorphisms to weak equivalences, and has good exactness properties.
Thus even if we start with a moduli problem without automorphisms, the derived problem leads us to consider simplicial sets. 
\end{enumerate}

\subsection*{Where do simplicial objects come from?}

Simplicial resolutions of schemes will be familiar to anyone who has computed \v Cech cohomology. Given a quasi-compact semi-separated scheme $Y$, 
we may take a finite affine cover $U=\coprod_i U_i$ of $Y$, and define the simplicial scheme $\check{Y}$ to be the \v Cech nerve $\check{Y}:= \cosk_0(U/Y)$. Explicitly,  
\begin{eqnarray*}
\check{Y}_n= \overbrace{U\by_{Y}U \by_{Y} \ldots  \by_{Y}U}^{n+1}
= \coprod_{i_0, \ldots, i_n} U_{i_0}\cap \ldots \cap U_{i_n},
\end{eqnarray*}
so $\check{Y}_n$ is an affine scheme, and $\check{Y}$ is the unnormalised \v Cech resolution of $Y$.

Given  a quasi-coherent sheaf $\sF$ on $Y$, we can then form a cosimplicial complex $\check{C}^n(Y, \sF):= \Gamma(\check{Y}_n, \sF)$, and of course Zariski cohomology is given by
\[
 \H^i(Y, \sF) \cong \H^i \check{C}^{\bt}(Y, \sF).      
\]
 
Likewise, if $\fY$ is a semi-separated Artin stack, we can choose a presentation $U \to \fY$ with $U$ affine, and set $\check{Y}:= \cosk_0(U /\fY)$, so
\[
 \check{Y}_n= \overbrace{U\by_{\fY}U \by_{\fY} \ldots  \by_{\fY}U}^{n+1}.       
\]
Resolutions of this sort were used by Olsson in \cite{olssartin} to study quasi-coherent sheaves on Artin stacks.

\subsection*{Questions}

\begin{enumerate}
        \item Which simplicial affine schemes correspond to schemes, Artin stacks or Deligne--Mumford stacks in this way?

\item What about higher stacks? 

[For an example of a higher stack, moduli of perfect complexes $\sF$ on $X$ will give an $n$-stack provided we restrict to complexes with $\Ext_X^i(\sF,\sF)=0 $ for all $i\le-n$. Similarly, a $2$-stack governs moduli of stacky curves.]

\item What about derived schemes and stacks?

\item What do quasi-coherent sheaves then look like?
\end{enumerate}

The first two questions will be addressed in Theorem \ref{bigthm}, and the third in Theorem \ref{dbigthm}. The fourth will be addressed in \S\S \ref{qcohsn} and \ref{dqcohsn}.

 I would like to thank Jo\~ao Pedro dos Santos for making many helpful suggestions and correcting several errors.




\section{Hypergroupoids}

\subsection{Simplicial sets}

\begin{definition}
 Define $|\Delta^n|$ to be the geometric $n$-simplex $\{ x \in \R_+^{n+1}\,:\, \sum_{i=0}^nx_i=1\}$. Write $\pd^i\co |\Delta^{n-1}|\to |\Delta^{n}|$ for the inclusion of the $i$th face, and $\sigma^i\co |\Delta^{n+1}|\to |\Delta^{n}| $ for the map given by collapsing the edge $(i, i+1)$.
\end{definition}

\begin{definition}
Given a topological space $X$, define $\Sing(X)_n$ to be the set of continuous maps from $|\Delta^n|$ to $X$. These fit into a diagram  
\[
\xymatrix@1{ \Sing(X)_0 \ar@{.>}[r]|{\sigma_0}& \ar@<1ex>[l]^{\pd_0} \ar@<-1ex>[l]_{\pd_1} \Sing(X)_1 \ar@{.>}@<0.75ex>[r] \ar@{.>}@<-0.75ex>[r]  & \ar[l] \ar@/^/@<0.5ex>[l] \ar@/_/@<-0.5ex>[l] 
\Sing(X)_2 &  &\ar@/^1pc/[ll] \ar@/_1pc/[ll] \ar@{}[ll]|{\cdot} \ar@{}@<1ex>[ll]|{\cdot} \ar@{}@<-1ex>[ll]|{\cdot}  \Sing(X)_3 & \ldots&\ldots,}
\]
where the arrows satisfy various relations such as $\pd_i\sigma_i= \id$. (Note that contravariance has turned superscripts into subscripts). 

Any diagram of this form is called a simplicial set. (For a rigorous definition, see \cite[\S 8.1]{W}) 
We will denote the category of simplicial sets by $s\Set$. We can define simplicial diagrams in any category similarly, while cosimplicial diagrams are given by reversing all the arrows.
\end{definition}

If $A_{\bt}$ is a simplicial abelian group, then note that setting $d:= \sum_{i=0}^n(-1)^i\pd_i \co A_n \to A_{n-1}$ gives maps satisfying $d^2=0$, so $A_{\bt}$ becomes a chain complex.

\begin{definition}
For $n \ge 0$, the combinatorial $n$-simplex $\Delta^n \in s\Set$ is characterised by the property that $\Hom_{s\Set}(\Delta^n, K) \cong K_n$ for all simplicial sets $K$. Its boundary $\pd\Delta^n \subset \Delta^n$ is given by $\bigcup_{i=0}^n\pd^i(\Delta^{n-1})$ (taken to include the case $\pd\Delta^0=\emptyset$), and for $n \ge 1$ the $k$th horn $\L^{n,k}$ is given by $\bigcup_{\substack{0 \le i \le n\\ i \ne k}}\pd^i(\Delta^{n-1})$.
\end{definition}

\begin{definition}
There is a geometric realisation functor $|-|\co s\Set \to \Top$, left adjoint to $\Sing$. This is characterised by the properties that it preserves colimits and that $|\Delta^n|=|\Delta^n|$.
\end{definition}

Draw a picture of $|\L^{2,k}|$ or $|\L^{3,k}|$ and you will see the reasoning for both the term horn and the notation $\L$.

\begin{definition}
 A map $f\co K \to L$ in $s\Set$ is said to be a weak equivalence if $|f|\co |K| \to |L|$ is a weak equivalence (i.e. induces isomorphisms on all homotopy groups).
\end{definition}
Note that the canonical maps $|\Sing(X)|\to X$ and hence $K \to \Sing(|K|)$ are always weak equivalences, so $s\Set$ and $\Top$ have the same homotopy theory (\cite[Theorem 11.4]{sht}). 

\subsection{Hypergroupoids}

\begin{definition}\label{duskindef}
\cite{duskin, glenn}: A (Duskin--Glenn) $n$-hypergroupoid (often also called a weak $n$-groupoid) is an object $X \in s\Set$ for which the partial matching maps
$$
X_m= \Hom_{s\Set}(\Delta^m, X) \to \Hom_{s\Set}(\L^{m,k}, X) 
$$
are surjective for all $m,k$,   and isomorphisms for all $m> n$ and all $k$. 
\end{definition}
The first condition is equivalent to saying that $X$ is a Kan complex, or fibrant. 

\begin{examples}
\begin{enumerate}
\item
A $0$-hypergroupoid is just a set $X=X_0$. 

\item
\cite[\S 2.1]{glenn} (see also \cite[Lemma I.3.5]{sht}): Every $1$-hypergroupoid $X$ arises as the nerve $B\Gamma$ of some groupoid $\Gamma$, so is given by
\[
(B\Gamma)_n= \coprod_{x_0, \ldots, x_n}\Gamma(x_0,x_1)\by \Gamma(x_1,x_2)\by\ldots  \by \Gamma(x_{n-1},x_n).       
\]
 Moreover, $\Gamma$ can be recovered from $X$  by taking objects $X_0$, morphisms $X_1$, source and target $\pd_0, \pd_1: X_1 \to X_0$, identity $\sigma_0 : X_0 \to X_1$ and multiplication
$$
X_1 \by_{\pd_0, X_0, \pd_1} X_1 \xra{(\pd_2,\pd_0)^{-1}} X_2 \xra{\pd_1} X_1.
$$
Equivalently, $\Gamma$ is the fundamental groupoid $\pi_fX$ of $X$ (as in \cite[\S I.8]{sht}).

\item Under the Dold--Kan correspondence between non-negatively graded chain complexes and simplicial abelian groups (\cite[\S 8.4]{W}), $n$-hypergroupoids in abelian groups correspond to chain complexes concentrated in degrees $[0,n]$. 
\end{enumerate}
\end{examples}

\begin{properties}\label{hgpdprops}
\begin{enumerate}
 \item  For an $n$-dimensional hypergroupoid $X$,  $\pi_mX=0$ for all $m>n$.

\item \label{truncate}\cite[Lemma \ref{stacks-truncate}]{stacks2}: An $n$-hypergroupoid $X$ is completely determined by its truncation $X_{\le n+1}$. Explicitly, $X= \cosk_{n+1}X$, where the $m$-coskeleton $\cosk_mX$ is given by  $(\cosk_mX)_i= \Hom((\Delta^i)_{\le m}, X_{\le m})$.
Moreover, a simplicial set of the form $\cosk_{n+1}X$ is an $n$-hypergroupoid if and only if  it satisfies the conditions of Definition \ref{duskindef} up to level $n+2$. 

When $n=1$, these statements amount to saying that a groupoid is uniquely determined by its objects (level $0$), morphisms and identities (level $1$) and multiplication (level $2$). However, we do not know we have a groupoid until we check associativity (level $3$). 
\end{enumerate}
\end{properties}

There is also a notion of relative $n$-hypergroupoids $X \to Y$, expressed in terms of (unique) liftings of
\[
 \xymatrix{ \L^{m,k}\ar[d] \ar[r] & X \ar[d]\\
\Delta^m \ar@{.>}[ur]\ar[r] & Y.
}       
\]

For example, a relative $0$-dimensional hypergroupoid $f\co X \to Y$ is a Cartesian morphism, in the sense that the maps
$$
X_n\xra{(\pd_i, f)} X_{n-1}\by_{Y_{n-1}, \pd_i}Y_n
$$
are all isomorphisms. Given $y \in Y_0$, we can write $F(y):= f_0^{-1}\{y\}$, and observe that $f$ is equivalent to a local system on $Y$ with fibres $F$.
The analogue of Property \ref{hgpdprops}.\ref{truncate} above also holds in the relative case. Level $0$ gives us the fibres $F(y)$, level $1$ gives us the descent data $\theta(z)\co F(\pd_0z) \cong F(\pd_1z)$ for $z \in Y_1$ (thereby determining the local system uniquely), but  we do not know we have a groupoid until we check the cocycle condition (level $2$):  $\theta(\pd_2w)\circ \theta(\pd_0w) = \theta(\pd_1w)$ for all $w \in Y_2$.      


\section{Higher stacks}

We will now show how to develop the theory of higher Artin stacks. For other types of stack, just modify the notion of covering. In particular, for Deligne--Mumford stacks, replace ``smooth'' with ``\'etale'' throughout. For simplicity of exposition, we will assume that everything is quasi-compact, quasi-separated etc. (strongly quasi-compact in the terminology of \cite{hag2}) --- to allow more general objects, replace affine schemes with arbitrary disjoint unions of  affine schemes. 

Given a simplicial set $K$ and a simplicial affine scheme $X$, there is an affine scheme $\Hom_{s\Set}(K, X)$ with the property that for all rings $A$, $\Hom_{s\Set}(K, X)(A)= \Hom_{s\Set}(K,X(A))$. Explicitly, when $K= \L^{m,k}$ this is given by the equaliser of a diagram
\[
  \prod_{\substack{0\le i \le m\\i \ne k}} X_{m-1} \implies    \prod_{\substack{0\le i<j \le m\\i,j \ne k}} X_{m-2}.     
\]

\begin{definition}\label{npdef}
Define an Artin $n$-hypergroupoid  to be a simplicial affine scheme $X_{\bt}$, such that the    partial matching maps
$$
X_m= \Hom_{s\Set}(\Delta^m, X) \to \Hom_{s\Set}(\L^{m,k}, X)
$$
are   smooth surjections   for all $k,m$, and isomorphisms  for all $m>n$ and all $k$.
\end{definition}

The idea of using such objects to model higher stacks is apparently originally due to Grothendieck, buried somewhere in \cite{pursuingstacks}.

\begin{remark}
Note that hypergroupoids can be defined in any category containing pullbacks along covering morphisms. In \cite{zhu},
this is used
to define Lie $n$-groupoids (taking the category of manifolds, with coverings given by submersions). A similar approach could be used to define higher topological stacks (generalising \cite{Noohi1}), taking surjective local fibrations as the coverings in the category of topological spaces. Similar constructions could be made in  non-commutative geometry and in synthetic differential geometry. 
\end{remark}

Fix a base ring $R$. Given any Artin $n$-hypergroupoid $X$ over  $R$, there is an associated functor $X\co \Alg_R \to s\Set$, given by $X(A)_n:= X_n(A)$. The following is \cite[Theorem \ref{stacks-relstrict}]{stacks2}:

\begin{theorem}\label{strict}
If $X$ is an Artin $n$-hypergroupoid $X$ over $R$, then its hypersheafification $X^{\sharp}\co \Alg_R \to s\Set$ is an $n$-geometric Artin stack in the sense of \cite[Definition 1.3.3.1]{hag2}. Every $n$-geometric Artin stack arises in this way. 
\end{theorem}

For set-valued functors, hypersheafification is just ordinary sheafification. For groupoid-valued functors, hypersheafification is stackification. An analogous definition can be made for $s\Set$-valued functors, but in the next section we will give an explicit description. 

\begin{remark}\label{cflurie}
Beware that
there are slight differences in terminology between \cite{hag2} and \cite{lurie}. In the former, only  affine schemes are $0$-representable, so arbitrary schemes might only be $2$-geometric, while Artin stacks are $1$-geometric stacks if and only if they have affine diagonal. In the latter, algebraic spaces are $0$-stacks.  

An $n$-stack $\fX$ in the sense of \cite{lurie} is called $n$-truncated in \cite{hag2}, and can be characterised by the property that that 
$
 \pi_i(\fX(A))=0
$
for all $i>n$ and $A \in \Alg_R$. For a hypergroupoid $X$, this amounts to saying that the associated stack $X^{\sharp}$ is $n$-truncated if and only if the maps
\[
X_i= \Hom_{s\Set}(\Delta^i,X) \to  \Hom_{s\Set}(\pd\Delta^i,X)
\]
are monomorphisms (i.e. immersions) of affine schemes for all $i>n$.

It follows easily from Property \ref{hgpdprops}.\ref{truncate} that every $n$-geometric stack in \cite{hag2} is $n$-truncated;
conversely,  any  $n$-truncated stack $\fX$ is $(n+2)$-geometric. Any Artin stack with affine diagonal (in particular any separated algebraic space) is $1$-geometric.

If we used algebraic spaces instead of  affine schemes in Definition \ref{npdef}, then Theorem \ref{strict} would adapt to give a characterisation of $n$-truncated Artin stacks. Our main motivation for using affine schemes as the basic objects is that they will be easier to translate to a derived setting.
\end{remark}

\section{Morphisms and equivalences}

Theorem \ref{strict} is all very well, but is clearly not the whole story. For a start, it gives us no idea of how to construct the hypersheafification. Thus we have no way of understanding morphisms between $n$-geometric stacks (as the hypersheafification is clearly not full), or even of knowing when two hypergroupoids will give us equivalent $n$-geometric stacks. If we think of the hypergroupoid as  analogous to the atlas of a manifold, then we need a notion similar to refinement of an open cover.

\subsection{Trivial relative hypergroupoids}

\begin{definition}\label{ptrel}
Say that a morphism    $f\co X\to Y$ of simplicial affine schemes is a    trivial  relative Artin (resp. Deligne--Mumford) $n$-hypergroupoid if the  
relative  matching maps
$$
X_m \to \Hom_{s\Set}(\pd \Delta^m,X)\by_{\Hom_{s\Set}(\pd \Delta^m,Y)}Y_m 
$$
are smooth (resp. \'etale) surjections for all $m\ge 0$, and isomorphisms for all $m\ge n$. 
\end{definition}

An example of a trivial  relative Artin $1$-hypergroupoid in stacks is the \v Cech nerve $\check{Y} \to \fY$ constructed in the introduction.

\begin{property}\label{ttruncate}
\cite[Lemma \ref{stacks-ttruncate}]{stacks2}:
 Trivial  relative $n$-hypergroupoids are completely determined by their truncations in levels $<n$. 
Explicitly, a morphism $f:X\to Y$ is a trivial  relative Artin (resp. Deligne--Mumford) $n$-hypergroupoid if and only if $X= Y\by_{\cosk_{n-1}Y}\cosk_{n-1}X$, and the $(n-1)$-truncated morphism $X_{<n}\to Y_{<n}$ satisfies the conditions of Definition \ref{ptrel} (up to level $n-1$). 
\end{property}

\subsection{Sheafification and morphisms}\label{sheaf}

\begin{theorem}\label{bigthm}
 The homotopy category of  strongly quasi-compact $n$-geometric  Artin  stacks is given by taking the full subcategory of $s\Aff$ consisting of  Artin  $n$-hypergroupoids, and formally inverting the trivial relative   Artin (resp. Deligne--Mumford) $n$-hypergroupoids.

In fact, a model for the $\infty$-category of strongly quasi-compact $n$-geometric   Artin  stacks is given by the relative category consisting of  Artin  $n$-hypergroupoids and the class of trivial relative Artin $n$-hypergroupoids.

The same results hold true if we substitute ``Artin'' with ``Deligne--Mumford'' throughout.
\end{theorem}

The relative categories of \cite{BarwickKanEquiv} give the cleanest description of the $\infty$-category  (as suggested to the author by the referee of \cite{stacks2}).  Before proving Theorem \ref{bigthm}, we will  give a more explicit description of the $\infty$-category.   

\begin{lemma}\label{sheaftriv}
If $f\co X\to Y$ is a    trivial  relative Artin $n$-hypergroupoid, then for all rings $A$, the map
$
 X^{\sharp}(A) \to Y^{\sharp}(A)       
$
is a weak equivalence in $s\Set$.
\end{lemma}

\begin{definition}
Define the simplicial $\Hom$ functor on    simplicial affine schemes by letting $\HHom_{s\Aff}(X,Y)$ be the simplicial set given by 
\[
\HHom_{s\Aff}(X,Y)_n:= \Hom_{s\Aff}(\Delta^n \by X, Y).        
\]
\end{definition}

\begin{definition}\label{hyperdef}
Given an  Artin $n$-hypergroupoid $Y$ and $X \in s\Aff$, define 
\[
       \HHom_{s\Aff}^{\sharp}(X,Y):= \LLim  \HHom_{s\Aff}(\tilde{X}, Y),
\]
where $\tilde{X} \to X$ runs over any  weakly initial filtered inverse system of  trivial  relative Artin $n$-hypergroupoids. [In fact, \cite[Corollary \ref{stacks-procofibrantet}]{stacks2} shows that there is a suitable inverse system $\tilde{X} \to X$ of   trivial  relative Deligne--Mumford $n$-hypergroupoids.] 
\end{definition}

The following is \cite[Corollary \ref{stacks-duskinmor}]{stacks2}:
\begin{theorem}\label{duskinmor}
If $X \in s\Aff$ and $Y$ is an  Artin $n$-hypergroupoid, then the derived $\Hom$ functor on  hypersheaves is given (up to weak equivalence)   by
\[
 \oR \HHom(X^{\sharp}, Y^{\sharp}) \simeq  \HHom_{s\Aff}^{\sharp}(X,Y).   
\]      
\end{theorem}

\begin{remarks}\label{higherrks}
 Given a ring $A$, set $X=\Spec A$, and note that $\HHom_{s\Aff}^{\sharp}(X,Y)= Y^{\sharp}(A)$, the hypersheafification of the functor $Y: \Alg_R \to s\Set$, so we can take Definition \ref{hyperdef} as a definition of sheafification for Artin hypergroupoids, giving an explicit description of $Y^{\sharp}$.
 The $n=1$ case should be familiar as the standard definition of sheafification.      

Between them,  Theorems \ref{strict} and \ref{duskinmor} recover the simplicial category $\cA_n$ of strongly quasi-compact $n$-geometric Artin stacks, with Theorem \ref{strict}
giving the objects and Theorem \ref{duskinmor} the morphisms. We could thus take those theorems to be a definition of that simplicial category. 

Moreover, \cite[Remark \ref{stacks-holimrk}]{stacks2} shows that 
\begin{align*}
 \pi_0\HHom_{s\Aff}^{\sharp}(X,Y)&\simeq \LLim_{X'}\pi_0\HHom_{s\Aff}(X',Y),\\
\pi_n(\HHom_{s\Aff}^{\sharp}(X,Y),f) &\simeq \LLim_{X' }\pi_n(\HHom_{s\Aff}(X',Y) ,f),
\end{align*}
where the limit is taken over the homotopy category of all trivial relative Deligne--Mumford $n$-hypergroupoids $X' \to X$.

Even for classical Artin or Deligne--Mumford stacks, this gives a shorter (and arguably more satisfactory) definition, since classical algebraic stacks  are just $1$-truncated (see Remark \ref{cflurie}) geometric stacks. For semi-separated algebraic spaces or schemes ($0$-truncated \'etale and Zariski $1$-hypergroupoids, respectively), this definition is at least comparable in complexity to the classical one. Note that making use of Properties \ref{hgpdprops}.\ref{truncate} and \ref{ttruncate} allows us  characterise all of these concepts in terms of finite diagrams of affine schemes.
\end{remarks}

\begin{proof}[Proof of Theorem \ref{bigthm}]
We just need to show that the functor $\HHom_{s\Aff}^{\sharp}$ on pairs of Artin (resp. Deligne--Mumford) $n$-hypergroupoids is the right-derived functor of $\Hom_{s\Aff}$ with respect to the class $\cT_n$ of trivial relative Artin (resp. Deligne--Mumford) $n$-hypergroupoids. In other words, we need to show that $\HHom_{s\Aff}^{\sharp}$ is the universal bifunctor under $\Hom_{s\Aff}$ mapping $\cT_n$ to isomorphisms  in $\Ho(s\Set)$. 

Lemma \ref{sheaftriv} ensures that $\HHom_{s\Aff}^{\sharp}$ sends $\cT_n$ to weak equivalences, so we need only prove universality. For this, we just observe that if $Y$ is an Artin (resp. Deligne--Mumford) $n$-hypergroupoid, then the map $Y \to Y^{\Delta^n}$ is the section of a morphism in $\cT_n$.  Since $\tilde{X} \to X$ is a system in $\cT_n$, this means that $\HHom_{s\Aff}^{\sharp}$ is constructed from $\Hom_{s\Aff} $ and $\cT_n$, 
so any bifunctor under $\Hom_{s\Aff}$ must lie under $\HHom_{s\Aff}^{\sharp} $ if it sends $\cT_n$ to weak equivalences.
\end{proof}

\subsection{Eilenberg--Mac Lane spaces and cohomology}

Given any abelian group $n$, the Dold--Kan correspondence allows us to form a simplicial abelian group $K(A,n)$ given by denormalising the chain complex $A[-n]$. Explicitly, $K(A,n)$ is freely generated under degeneracy operations $\sigma_i$ by a copy of $A$ in level $n$, so  $K(A,n)_m \cong A^{\binom{m}{n}}$ (see \cite[8.4.4]{W} for details).

Given a smooth commutative affine group scheme $A$, this construction gives us an Artin $n$-hypergroupoid $K(A,n)$ (in fact, if $n=1$, $A$ need not be commutative), and then for any Artin stack $\fX$,
\[
\H^{n-i}_{\et}(\fX, A) \cong   \pi_i\oR \HHom(X^{\sharp}, K(A,n)^{\sharp}).
\]
In particular, taking $A= \bG_a$ gives us $\H^*(\fX, \O_{\fX})$, while taking $A = \Z/\ell^m$ (regarded as a finite scheme) gives $\H^*_{\et}(\fX,\Z/\ell^m )$. 

We could generalise this by allowing $A$ to be a smooth commutative algebraic group space, in which case  $K(A,n)$ would be an Artin $n$-hypergroupoid in algebraic spaces, making $K(A,n)^{\sharp}$ an $n$-truncated geometric stack (see Remark \ref{cflurie}), and hence representable by an Artin $(n+2)$-hypergroupoid (in affine schemes). 

\section{Quasi-coherent sheaves}\label{qcohsn}

The following is \cite[Corollary \ref{stacks-qcohequiv}]{stacks2}:
\begin{proposition}
For an Artin $n$-hypergroupoid $X$, giving a quasi-coherent module on the $n$-geometric stack $X^{\sharp}$ is equivalent to giving
\begin{enumerate}
 \item a quasi-coherent sheaf $\sF^n$ on $X_n$ for each $n$, and
\item isomorphisms $\pd^i\co \pd_i^*\sF^{n-1} \to \sF^n$ for all $i$ and $n$, satisfying the usual cosimplicial identities. 
\end{enumerate}
\end{proposition}

Given a morphism $f\co X \to Y$ of Artin $n$-hypergroupoids, inverse images are easy to compute: we just have $(f^*\sF)^n:= f_n^*\sF^n$. Direct images are much harder to define, as taking $f_*$ levelwise destroys the Cartesian property. See \cite[\S \ref{stacks-directsn}]{stacks2} for an explicit description of the derived direct image functor $\oR f_*^{\cart}$.

\section{Derived stacks}

Motivated by the need for good obstruction theory and cotangent complexes, derived algebraic geometry replaces rings with simplicial rings. There is a normalisation functor $N\co s\Alg_R \to dg_+\Alg_R$ from simplicial $R$-algebras to commutative chain $R$-algebras in non-negative degrees. If $R$ is a $\Q$-algebra, then $N$ induces an equivalence on the homotopy categories (and also on the derived $\Hom$-spaces). [When $R$ is not a $\Q$-algebra, $dg_+\Alg_R$ is not even a model category.]

We will write $d\Alg$ for either of the categories $s\Alg_R, dg_+\Alg_R$, and write $d\Aff$ for the opposite category (derived affine schemes over $R$), denoting objects as $\Spec A$.
 Any derived ring $A \in d\Alg$ can be thought of as essentially an exotic nilpotent thickening of $\H_0A$, so there are equivalent variants of this theory replacing $A$ with its localisation (\cite[\S \ref{stacks-locthick}]{stacks2}),  its henselisation (\cite[\S \ref{stacks-henthick}]{stacks2}), or even (in Noetherian cases) its completion (\cite[Propositions \ref{stacks-fthm} and \ref{stacks-dgshrink}]{stacks2}) over $\H_0A$.

\begin{remark}
The constructions in this section will work for any model category with a suitable notion of coverings. In particular, they work for symmetric spectra, the basis of a theory known as topological, spectral, brave new or (unfortunately) derived algebraic geometry.
Roughly speaking (as explained in the introduction to \cite{lurieDAG5}), simplicial rings serve to apply homotopy theory to algebraic geometry, while symmetric spectra are used to do the opposite. For a detailed discussion, see \cite[\S 2.6]{lurie}. 
\end{remark}

\subsection{Derived hypergroupoids}

\begin{definition}
Say that a morphism in $s\Alg_R$ is quasi-free if it is freely generated in each level, with the generators closed under the degeneracy operations $\sigma_i$. Say that  a morphism in $dg_+\Alg_R$ is quasi-free if the underlying morphism of skew-commutative graded algebras is freely generated. 

Say that a morphism in $d\Alg$ is a cofibration if it is a retract of a quasi-free map, and that a morphism in $d\Aff$ is a fibration if it is $\Spec $ of a cofibration.
\end{definition}

Write $sd\Aff$ for the category of  simplicial derived affine schemes, i.e.  simplicial diagrams  in $d\Aff$. Weak equivalences in this category are maps $\Spec B \to \Spec A$ inducing isomorphisms $\H_*(B)\cong \H_*(A)$.

\begin{definition}
An object $X_{\bt} \in sd\Aff$ is said to be Reedy fibrant if the matching maps
\[
 X_n= \Hom_{s\Set}(\Delta^n, X) \to \Hom_{s\Set}(\pd \Delta^n, X)       
\]
are fibrations in $d\Aff$ for all $n$. 
\end{definition}

\begin{example}
Given a cofibration $R \to A$ in $s\Alg_R$, write $X= \Spec A$.  We may form an object $A\ten \Delta^n \in  s\Alg_R$ by 
\[
 (A\ten \Delta^n)_i := \overbrace{A_i\ten_RA_i \ten_R\ldots \ten_R A_i}^{\Delta^n_i}.
\]
 Then the simplicial derived affine scheme $\uline{X}$ given by $\uline{X}_n := \Spec ( A\ten \Delta^n)$ is Reedy fibrant, and $X \to \uline{X}$ is a weak equivalence levelwise. 

$\uline{X}$ will be familiar to some readers from the construction of simplicial $\Hom$, since by definition
\[
 \HHom_{s\Alg_R}(A,B) = \uline{X}(B) \in s\Set       
\]
for $A \in s\Alg_R$. More generally, for any $A \in s\Alg_R$, we can always take a quasi-isomorphism $A' \to A$ for $A'$ cofibrant, and then set
\[
 \oR\HHom_{s\Alg_R}(A,B) :=   \HHom(A',B).     
\]

There are analogous constructions in $dg_+\Alg_R$, but they are not so easily described.
 \end{example}

\begin{definition}
We say that a morphism $\Spec B \to \Spec A$ in $d\Aff$ is a smooth (resp. \'etale) surjection if $\Spec \H_0B \to \Spec \H_0A $ is so, and the maps
\[
 \H_i(A)\ten_{\H_0(A)}\H_0(B) \to \H_i(B)       
\]
 are all isomorphisms.
 \end{definition}

\begin{definition}\label{dnpdef}
A  derived Artin $n$-hypergroupoid is  a Reedy fibrant  object $X_{\bt} \in sd\Aff$ for which the  partial matching maps
$$
X_m= \Hom_{s\Set}(\Delta^m, X) \to \Hom_{s\Set}(\L^{m,k} X)
$$
are smooth surjections for all $m,k$ (i.e. all $m \ge 1$ and all $0 \le k \le m$), and weak equivalences for all $m>n$ and all $k$.
\end{definition}

\begin{remark}
Given any  derived Artin $n$-hypergroupoid $X$, we may form a simplicial scheme $\pi^0X$ by setting 
\[
 \pi^0X_n:= \Spec (\H_0 O(X_n)) \in \Aff.
\]
 Then observe that $\pi^0X$ is an  Artin $n$-hypergroupoid, equipped with a map $\pi^0X \to X$. We call this the underived part of $X$.       
\end{remark}

\subsection{Derived Artin stacks}

The following is \cite[ Theorem \ref{stacks-relstrict}]{stacks2}:
\begin{theorem}\label{strictd}
If $X$ is a derived  Artin $n$-hypergroupoid $X$ over $R$, then its hypersheafification $X^{\sharp}\co d\Alg_R \to s\Set$ is an $n$-geometric derived  Artin stack in the sense of \cite[Definition 1.3.3.1]{hag2}. Every strongly quasi-compact $n$-geometric derived Artin stack arises in this way. 
\end{theorem}

\begin{remarks}\label{cflurie2}
As with Remarks \ref{cflurie}, there is a difference in terminology between \cite{lurie} and \cite{hag2}. A geometric derived Artin $\infty$-stack $\fX$ is called an $n$-stack (Lurie) or $n$-truncated (To\"en--Vezzosi) if the associated underived Artin $\infty$-stack $\pi^0\fX$ is $n$-truncated. This implies that for $A \in d\Alg$ with $\H_i(A)=0$ for all $i>m$, we have $
 \pi_i(\fX(A))=0$
for all $i>m+n$.
\end{remarks}

\subsection{Morphisms and equivalences}\label{dmor}

\begin{definition}
A  trivial relative  derived Artin $n$-hypergroupoid is a morphism  $X_{\bt} \to Y_{\bt}$ in $sd\Aff$ for which the 
matching maps
$$
X_m \to \Hom_{s\Set}(\pd \Delta^m,X)\by_{\Hom_{s\Set}(\pd \Delta^m,Y )}Y_m 
$$
are smooth surjective fibrations for all $m \ge 0$ and weak equivalences for $m \ge n$.
\end{definition}

The results of \S \ref{sheaf} all now carry over:

\begin{theorem}\label{dbigthm}
The homotopy category of  strongly quasi-compact $n$-geometric derived Artin stacks is given by taking the full subcategory of $sd\Aff$ consisting of derived Artin $n$-hypergroupoids, and formally inverting the trivial relative  derived Artin $n$-hypergroupoids.

In fact, a model for the $\infty$-category of strongly quasi-compact $n$-geometric derived Artin stacks is given by the relative category consisting of derived Artin $n$-hypergroupoids and the class of trivial relative  derived Artin $n$-hypergroupoids.
\end{theorem}

Before proving this, we give the preliminary results necessary.  We can first define $\Hom$-spaces $\HHom_{sd\Aff}$ by $\HHom_{sd\Aff}(X,Y)_n= \Hom_{sd\Aff}(X,Y^{\Delta^n})$, and then  
the following is \cite[Theorem \ref{stacks-duskinmor}]{stacks2}:
\begin{theorem}\label{duskinmord}
If $X \in sd\Aff$ and $Y$ is a derived   Artin $n$-hypergroupoid, then the derived $\Hom$ functor on hypersheaves is given (up to weak equivalence) by
\[
 \oR \HHom(X^{\sharp}, Y^{\sharp}) \simeq \LLim  \HHom_{sd\Aff}(\tilde{X}, Y),      
\]
where $\tilde{X} \to X$ runs over any weakly initial filtered inverse system of   trivial  relative derived Deligne--Mumford $n$-hypergroupoids.      
\end{theorem}

\begin{remark}\label{derivedrk}
Given a derived  $R$-algebra $A$, set $X=\Spec A$, and note that $Y^{\sharp}(A)\simeq \oR \HHom(X^{\sharp}, Y^{\sharp})$, so the theorem gives an explicit description of $Y^{\sharp}$. In fact, we can take  the theorem to be  a definition of hypersheafification, and even as a definition of the simplicial category $\cD\cA_n$ of strongly quasi-compact $n$-geometric derived  Artin stacks (with Theorem \ref{strictd} giving the objects and Theorem \ref{duskinmord} the morphisms). 

Again, \cite[Remark \ref{stacks-holimrk}]{stacks2} shows that 
\begin{align*}
 \pi_0\HHom_{s\Aff}^{\sharp}(X,Y)&\simeq \LLim_{X'}\pi_0\HHom_{s\Aff}(X',Y),\\
\pi_n(\HHom_{s\Aff}^{\sharp}(X,Y),f) &\simeq \LLim_{X' }\pi_n(\HHom_{s\Aff}(X',Y) ,f),
\end{align*}
where the limit is taken over the homotopy category of all trivial relative derived Deligne--Mumford $n$-hypergroupoids $X' \to X$.

Note that the proof of Theorem \ref{dbigthm} is now exactly the same as that for Theorem \ref{bigthm}, replacing Theorem \ref{duskinmor} with Theorem \ref{duskinmord}.
\end{remark}

\subsection{Quasi-coherent complexes}\label{dqcohsn}

Since the basic building blocks for derived algebraic geometry are simplicial rings or chain algebras, the correct analogue of quasi-coherent sheaves has to involve complexes.

\begin{definition}
Given a chain algebra $A$, an $A$-module $M$ in complexes is a (possibly unbounded) chain complex $M$ equipped with a distributive chain morphism $A \ten M \to M$. Given a simplicial ring $A$, we just define an $A$-module in complexes to be an $NA$-module in complexes, where $NA$ is the chain algebra given by Dold--Kan normalisation.
\end{definition}

The following is \cite[Proposition \ref{stacks-qcohequiv}]{stacks2}:
\begin{proposition}\label{qcohequiv}
For a derived Artin $n$-hypergroupoid $X$, giving a quasi-coherent complex (in the sense of \cite[\S 5.2]{lurie})
on the $n$-geometric derived stack $X^{\sharp}$ is equivalent (up to quasi-isomorphism) to giving
\begin{enumerate}
 \item an $O(X_n)$-module $\sF^n$ in complexes for each $n$, and
\item quasi-isomorphisms $\pd^i\co \pd_i^*\sF^{n-1} \to \sF^n$ for all $i$ and $n$, satisfying the usual cosimplicial identities. 
\end{enumerate}
\end{proposition}

In broad terms, quasi-coherent complexes correspond to complexes $\sF_{\bt}$ of \emph{presheaves} of $\O_X$-modules  whose homology \emph{presheaves}   $\sH_n(\sF_{\bt})$ are quasi-coherent.  
To understand how these are related to complexes of quasi-coherent sheaves on schemes, see \cite[Remarks \ref{stacks-hcartrks}]{stacks2}.  

As in \S \ref{qcohsn}, inverse images 
of quasi-coherent complexes are easy to compute, while derived direct images 
are more complicated --- see \cite[\S \ref{stacks-directsn}]{stacks2}.

%

\subsection{Derived schemes and algebraic spaces}

\begin{definition}
 An $n$-geometric derived Deligne--Mumford stack $X$ is called a derived algebraic space (resp. derived scheme) if the associated underived $n$-stack $\pi^0X$ is represented by an algebraic space (resp. a scheme). For this to occur, we must have $n\le 2$ (or $n \le 1$ if $\pi^0X$ is semi-separated).
\end{definition}

 The following is given by \cite[Theorems \ref{stacks-lshfthm}, \ref{stacks-hshfthm} and \ref{stacks-dgshfthm}]{stacks2}:
\begin{theorem}\label{lshfthm}
Fix a  scheme $Z$ over a ring $R$. Then the  homotopy category of  derived schemes $X$ over $R$  with $\pi^0X \simeq Z$ is weakly equivalent to the   homotopy category  of  presheaves $\sA_{\bt}$ of derived $R$-algebras on  the category  of affine open subschemes of $Z$,  satisfying the following:
\begin{enumerate}
\item $\H_0(\sA_{\bt})= \O_{Z}$;
\item for all $i$, the presheaf $\H_i(\sA_{\bt})$ is a quasi-coherent $\O_{Z}$-module.
\end{enumerate}

The same result holds if we replace schemes with algebraic spaces, and open immersions with \'etale maps.
\end{theorem}

In this setting, Proposition \ref{qcohequiv}  becomes:

\begin{proposition}\label{qcohequiv2}
Take a derived scheme $X$, given by a pair $(Z,\sA_{\bt})$ as in Theorem \ref{lshfthm}. Then giving a quasi-coherent complex 
on $X$ is equivalent (up to quasi-isomorphism) to giving a presheaf $\sF_{\bt}$ of $\sA_{\bt}$-modules in complexes for which the presheaves $\H_i(\sF_{\bt})$ are all quasi-coherent as $\O_Z$-modules.
\end{proposition}

\bibliographystyle{alphanum}
\bibliography{references}
\end{document}